\documentclass[article]{amsart}
\usepackage{cases}
\usepackage{amsmath, amsfonts, pifont, amssymb}
\usepackage{verbatim}
\usepackage[bookmarks=true]{hyperref}
\usepackage{geometry}
\newtheorem{theorem}{Theorem}[section]
\newtheorem{lemma}[theorem]{Lemma}
\newtheorem{proposition}[theorem]{Proposition}

\newcommand{\R}{\mathbb{R}}

\newcommand{\beq}{\begin{equation}}
\newcommand{\eeq}{\end{equation}}
\newcommand{\beqq}{\begin{equation*}}
\newcommand{\eeqq}{\end{equation*}}
\newcommand{\less}{\lesssim}
\newcommand{\gea}{\gtrsim}

\newcommand{\lf}{\left}
\newcommand{\ri}{\right}
\theoremstyle{definition}

\theoremstyle{remark}
\newtheorem{remark}[theorem]{Remark}

\numberwithin{equation}{section}

\numberwithin{equation}{section}
\usepackage[notcite, notref]{showkeys}

\DeclareMathOperator{\supp}{\mathrm{supp}}

\begin{document}
\title{On a bilinear restriction estimate for Schr\"odinger equations on 2D waveguide}
\author{Deng Yangkendi}

\address{Kendi Dengyang
\newline \indent Academy of Mathematics and Systems Science, Chinese Academy of Sciences, Beijing, China.}
\email{dengyangkendi@amss.ac.cn}
\maketitle
\begin{abstract}
In this article, we prove a bilinear estimate for Schr\"odinger equations on 2d waveguide, $\mathbb{R}\times \mathbb{T}$. We hope it may be of use in the further study of concentration compactness for cubic NLS on  $\mathbb{R}\times \mathbb{T}$.
\end{abstract}

\section{Introduction}
\subsection{Statement of main results}
In this article, we consider linear Schr\"odinger equations on the waveguide $\mathbb{R}\times \mathbb{T}$,
\begin{equation}\label{eq: ls}
\begin{cases}
iu_{t}+\Delta u=0,\\
u(0,x)=u_{0}.	
\end{cases}
\end{equation}
and prove the following bilinear estimate,
\begin{theorem}\label{thm: main}
Suppose $1 \ll \delta \ll T^{-1},  p=\frac{12}{7}$ and functions $f, g$ in $L^2(\mathbb{R} \times \mathbb{T})$ satisfy $supp(\widehat{f})\subset \theta_{1}, supp(\widehat{g})\subset \theta_{2} $,  where $\theta_{1}, \theta_{2} \in \mathcal{C_\delta}$ and $dist(\theta_{1}, \theta_{2})\gg \delta$. Then
\begin{equation}\label{eq: main}
\| e^{it\Delta}f \cdot  e^{it\Delta} g    \|_{L^2([0,T] \times \mathbb{R} \times \mathbb{T})}\lesssim  \delta^{2-\frac{4}{p}}   \|\widehat{f} \|_{L^p(\mathbb{R} \times \mathbb{Z})} \|\widehat{g} \|_{L^p(\mathbb{R} \times \mathbb{Z})} .
\end{equation}

 where for $\delta \in 2^\mathbb{N}$, we define
 \begin{equation}
 \mathcal{C}_{\delta}:=\{\theta :\theta=[m\delta,(m+1)\delta]\times \{n\delta, n\delta+1, \cdots, (n+1)\delta\}\subset \mathbb{R} \times \mathbb{Z}, m,n \in \mathbb{Z}  \}.
 \end{equation}
\end{theorem}

It is worth noting that the condition $T\delta \ll 1$ is necessary when $p=\frac{12}{7}$.

\subsection{Background, motivation and some discussions}
In the setting of linear Schr\"odinger equation on waveguide $\mathbb{R}\times \mathbb{T}$, estimate
\eqref{eq: main} is a natural but partial generalization of a bilinear restriction estimate in \cite{moyua1999restriction}, see also the $n=3$ case of Theorem 2.3 in \cite{tao1998bilinear}.  Note that unlike the Schr\"odinger equation on $\mathbb{R}^{2}$, \eqref{eq: ls} does not have the full scaling symmetry, and does not admit the usual global in time Strichartz estimate, thus one needs the parameter $\delta$ and $T$ in the statement of \eqref{eq: main}. The constraint $\delta\gg 1$ is harmless, but assuming $\delta\ll T^{-1}$ is unfavorable, since in some sense, this means that the associated Schr\"odinger wave travels only $T\delta\ll 1$. However, a direct generalization of Theorem 2.3 (in case $n=3$) in \cite{tao1998bilinear} to our setting does not hold, and we will discuss a counter-example in the appendix. It should be noted that due to our constraint $T\delta\ll 1$, the estimate in some sense falls into the so-called semiclassical regime, thus it is possible that one can extend this result to more general manifolds, see for example \cite{burq2004strichartz},\cite{hani2012bilinear} and reference therein for this direction.

 Unlike the bilinear estimates in \cite{fan2016on}, \cite{zhao2021Long}, the bilinear estimates of form \eqref{eq: main} are closely related to the concentration compactness theory to the mass critical NLS. As well explained in \cite{begout2007mass}, after pioneering work of \cite{bourgain1998refinements}, \cite{merle1998compactness}, in $\mathbb{R}^{n}$, one can systematically transfer bilinear restriction type estimates from \cite{tao2003bilinear},      
\begin{equation}\label{eq: bilinearestimat}
\| e^{it\Delta_{\R^n}}f \cdot  e^{it\Delta_{\R^n}} g    \|_{L^r(\R \times \mathbb{R}^n)}\lesssim  \|\widehat{f} \|_{L^{p}(\mathbb{R}^n)} \|\widehat{g} \|_{L^{p}(\mathbb{R}^n)},
\end{equation}
where $r=\frac{n+2}{n}, \frac{2}{p^\prime}>\frac{n+3}{n+1} \frac{1}{r}$,  into 
\begin{equation}\label{eq: xpqestimate}	
\|e^{it\Delta_{\R^n}}f\|_{L^q(\R \times \mathbb{R}^n)} \lesssim \|\widehat{f}\|_{X^{p,q}},
\end{equation}
where $q=\frac{2(n+2)}{2}$ and the space $X^{p,q}$ is defined by
$$ \|g\|_{X^{p,q}}= \left(\sum_{\delta \in 2^\mathbb{Z}} \delta^{\frac{n}{2}\frac{p-2}{p}q}\sum_{\theta \text{ is a dyadic cube of sidelength } \delta} \| g \|_{L^p(\theta)}^q  \right)^{\frac1q} .$$
The $X^{p,q}$ space estimate will leads a to concentration compactness result, which roughly says the following. \\
\textbf{One rough statement of concentration compactness}:
Let $f_{n}$ be bounded in $L^{2}(\mathbb{R}^{d})$, with
\begin{equation}
\limsup_{n\rightarrow \infty}\|e^{it\Delta}f_{n}\|_{L_{t}^{p}L_{x}^{q}(\mathbb{R}\times \mathbb{R}^{d})}>0,
\end{equation}
then $f_{n}$ (up to extracting subsequence), must concentrates at certain scale, associated with the symmetries of the linear Schr\"odinger equation (time translation, space translation, Galilean, scaling).

We refer to \cite{begout2007mass}, \cite{bourgain1998refinements}, \cite{merle1998compactness} for more details. Ever since the work of \cite{kenig2006global},\cite{kenig2008global}, concentration compactness techniques have been widely applied in the study of nonlinear dispersive PDEs. Furthermore, after the work of \cite{ionescu2012energy},\cite{ionescu2010global}, it has been understood that a concentration compactness analysis can systematically transfer a scattering result in Euclidean space to a global wellposedness result in a manifold.

Our motivation to study estimates of form \eqref{eq: main}, is to explore the concentration compactness for defocusing cubic NLS on $\mathbb{R}\times \mathbb{T}$.

 Note that this equation is special in the sense it is mass critical, thus to study its wellposedness at critical regularity $L^{2}$ requires a no derivative loss Strichartz estimate. Such an estimate is available in $\mathbb{R}\times \mathbb{T}$, \cite{takaoka20012d}, but simply wrong in 2D Tori. There is also a global in time Strichartz-type estimate on $\mathbb{R}\times \mathbb{T}$, \cite{barron2021Global} and we refer to \cite{barron2021ON} for the global in time Strichartz estimate on the more general waveguide.   On the other hand, one should expect that, after the work \cite{takaoka20012d}, if one can establish a concentration compactness result (for linear Schr\"odinger equation) on $\mathbb{R}\times \mathbb{T}$, then one can transfer Dodson's scattering result \cite{dodson2015global} in $\mathbb{R}^{2}$ to establish GWP for defocusing cubic NLS on $\mathbb{R}\times \mathbb{T}$.

\subsection{Proof Strategy}
We mainly follow \cite{tao1998bilinear}, using interpolation with analytic families, \cite{stein1956interpolation}, to reduce the study of an $L^{2}$ estimate.  By almost orthogonality principle, and in the setting of \cite{tao1998bilinear}, it will follow from certain non-stationary phase analysis.

Here, we will need some analytic number theory argument to replace the usual integration by parts argument in real line, since the analysis on $\mathbb{T}$ will natural introduce certain study of exponential sums.

\subsection{Notation}
We use $A\less_s B$ to denote an estimate of the form $A\le CB$ for a constant $C$ which depends on some parameter $s$. Say $A\sim_s B$, if $A\less_s B$ and $A\gea_s B$. We write $A \ll B$ if there exists a small constant $c>0$ such that $A\le cB$.

We usually consider a function $F$ defined on $\R \times \mathbb{Z}$. For simplicity, we denote
$$ \int_{\R \times \mathbb{Z}} F(y) {\rm d} y := \sum_{n\in \mathbb{Z}}\int_{\R} F(\xi,n) {\rm d} \xi.$$

For a function $f$ on $\R \times \mathbb{T}$, we define the Fourier transform of $f$ by
$$ \widehat{f}(\xi,n)= \int_{\R \times \mathbb{T}} e^{-2\pi i (x_1 \xi + x_2 n)} f(x_1,x_2) {\rm d} x_1 {\rm d} x_2.$$

We write $e^{it \Delta} $ for the solution operator to the linear Schr\"odinger equation
$$ iu_{t}+\Delta u=0$$
on $\R \times \mathbb{T}$, that is
$$ e^{it \Delta} f (x)= \int_{\R \times \mathbb{Z}} e^{2\pi i (x\cdot y- 2\pi t |y|^2 )} \widehat{f}(y) {\rm d} y.$$

\subsection{Acknowledgement}
Thanks Prof.Zehua Zhao and Prof.Chenjie Fan for a lot of discussions for NLS on waveguide $\mathbb{R}^{n}\times \mathbb{T}^{m}$, and also proposes the study of NLS on $\mathbb{R}\times \mathbb{T}$, which leads to study in current article. This research has been partially supported by the CAS Project for Young Scientists in Basic Research, Grant No. YSBR-031.
\section{Reduction of Theorem \ref{thm: main} into an $L^{2}$ estimate}
In this section, we will calculate $\| e^{it\Delta}f \cdot  e^{it\Delta} g    \|_{L^2([0,T] \times \mathbb{R} \times \mathbb{T})}^2$ directly, and reduce the estimate (\ref{eq: main}) into an $L^{2}$ estimate. The calculation in this section is standard, and we refer to \cite{tao1998bilinear},\cite{barron2021Global} for similar discussions.

Let $\phi\in \mathcal{S}(\R) $ be a nonnegative radial decreasing Schwartz function with $\supp(\widehat{\phi})\subset [-1,1]$ and  $\phi \ge 1_{[-1,1]}$. We have
\begin{align*}
&\| e^{i t \Delta} f \cdot  e^{i t\Delta} g    \|_{L^2([0,T] \times \R \times \mathbb{T})}^2 \\
\le&\|\sqrt{\phi(t/T)} e^{i t\Delta} f \cdot e^{i t\Delta} g    \|^2_{L^2(\R \times \R \times \mathbb{T})} \\
=&\int_\R \int_{\R \times \mathbb{T}} \phi(\frac{t}{T}) e^{i t\Delta} f \cdot e^{i t\Delta} g \cdot \overline{e^{i t\Delta} f} \cdot \overline{e^{i t\Delta} g} \\
=& T  \int_{(\R\times \mathbb{Z})^4} \widehat{\phi}\lf(T Q(x,y,z,w) \ri)\delta_0(x+y-z-w) \widehat{f}(x) \overline{\widehat{f}(z)}  \widehat{g}(y) \overline{\widehat{g}(w)}   {\rm d} x {\rm d} y {\rm d} z {\rm d} w,
\end{align*}
where $\delta_0$ is the dirac measure, and
$$ Q(x,y,z,w)=|x|^2+|y|^2-|z|^2-|w|^2.$$
Let $\chi_1, \chi_2$ be the smooth cut-off functions of $\theta_1, \theta_2$ and define
$$ \chi(x,y,z,w)= \chi_1(x) \chi_1(z) \chi_2(y) \chi_2(w).$$
From the positivity of the kernel, and $|\widehat{\phi}(\cdot)|\less \phi(\frac12 \cdot)$, we may consider the 4-linear functional
\begin{align*}
& L(F_1,F_2,G_1,G_2) \\
:=& T  \int_{(\R\times \mathbb{Z})^4} \phi\lf(\frac12 T Q(x,y,z,w) \ri)\delta_0(x+y-z-w) \chi(x,y,z,w) \\
& \quad \cdot F_1(x) F_2(z)  G_1(y) G_2(w)   {\rm d} x {\rm d} y {\rm d} z {\rm d} w .
\end{align*}
To prove (\ref{eq: main}), it suffices to show that
\begin{equation}\label{eq:main-2}
|L(F_1,F_2,G_1,G_2)|\less \delta^{-\frac23} \|F_1\|_{L^{\frac{12}{7}}}\|F_2\|_{L^{\frac{12}{7}}}\|G_1\|_{L^{\frac{12}{7}}}\|G_2\|_{L^{\frac{12}{7}}}.
\end{equation}
One has
$$ |L(F_1,F_2,G_1,G_2)|\less \|F_1\|_{L^\infty} \|F_2\|_{L^1} \sup_{z \in \theta_1} |\widetilde{L}_{z}(G_1,G_2)|, $$
where
\begin{align*}
 & \widetilde{L}_{z}(G_1,G_2) \\
 =& T \int_{(\R\times \mathbb{Z})^3} \phi\lf(\frac12 T Q(x,y,z,w) \ri)\delta_0(x+y-z-w)\chi(x,y,z,w)   G_1(y) G_2(w)     {\rm d} x {\rm d} y {\rm d} w \\
 =& T \int_{(\R \times \mathbb{Z})^2}  \phi\lf( T (y-z)\cdot (y-w) \ri)\widetilde{\chi}(y,w) G_1(y) G_2(w) {\rm d} y {\rm d} w,
\end{align*}
where $\widetilde{\chi}(y,w)=\chi_{1}(w-y+z)\chi_{2}(y)\chi_{2}(w)$.

We may assume without loss of generality that $z=(0,0)\in \theta_1$. Consider
\begin{align*}
 \widetilde{L}(G_1,G_2)= T \int_{(\R \times \mathbb{Z})^2}  \phi\lf( T y\cdot (y-w) \ri)\widetilde{\chi}(y,w) G_1(y) G_2(w) {\rm d} y {\rm d} w.
\end{align*}
It suffices to prove that
\begin{equation}\label{eq:L}
|\widetilde{L}(G_1,G_2)|\less \delta^{-\frac23} \|G_1\|_{L^\frac32} \|G_2\|_{L^\frac32},
\end{equation}
because when we have the estimate
$$|L(F_1,F_2,G_1,G_2)| \less \delta^{-\frac23} \|F_1\|_{L^\infty} \|F_2\|_{L^1}  \|G_1\|_{L^\frac32} \|G_2\|_{L^\frac32}, $$
then we get (\ref{eq:main-2}) by symmetries of $F_1,F_2,G_1,G_2$ and multi-linear interpolation (or the H\"older inequality).

Now, we define
$$ U(G)(y)=T \delta^{\frac23} \int_{\R \times \mathbb{Z}}  \phi\lf( T y\cdot (y-w) \ri)\widetilde{\chi}(y,w) G(w){\rm d} w,$$
so (\ref{eq:L}) reduces to
$$ \|U(G)\|_{L^3} \less \|G\|_{L^{\frac32}}.$$
Fix a $\eta \in C_c^\infty(\R)$ with $\eta(s)=1$ when $|s|\le \frac{1}{100}$. Define
$$ a_\alpha(s)=e^{\alpha^2} \frac{s_+^{\alpha-1}}{\Gamma(\alpha)}\eta(s),$$
where $\alpha \in \mathbb{C}$.
We imbed the operator $U$ in the analytic family $U_\alpha$ defined by
$$ U_\alpha(G)(y)=T\delta^{\frac43 \alpha +\frac23} \int_{\R \times \mathbb{Z}} \lf(\phi(T \delta^2 \cdot) \ast a_\alpha (\cdot) \ri)(\delta^{-2} (y-w)\cdot y ) \widetilde{\chi}(y,w) G(w) {\rm d} w.$$
The result we need will follow from
\begin{equation}\label{bilinear estimate-thm-eq-2}
\|U_{1+it}G\|_{L^\infty} \less \|G\|_{L^1},
\end{equation}
and
\begin{equation}\label{bilinear estimate-thm-eq-3}
\|U_{-\frac{1}{2}+it}G\|_{L^2} \less \|G\|_{L^2}
\end{equation}
uniformly for all real $t$.

It is easy to prove (\ref{bilinear estimate-thm-eq-2}).  So it remains to prove the $L^2$ estimate (\ref{bilinear estimate-thm-eq-3}). It is well known that $\widehat{a}_{-\frac{1}{2}+it}$ has order $\frac12$, i.e., $|\widehat{a}_{-\frac{1}{2}+it}^{(N)}(s)|\less_N (1+|s|)^{\frac12-N}$ for any $N\in \mathbb{N}$ and uniformly for $t\in \R$. For more properties about the function $a_\alpha$, we refer to \cite{stein1993Harmonic} Chapter 9, \S 1.2.3.

Now we write  $U_{-\frac12+it}$ as
\begin{align*}
U_{-\frac{1}{2}+it}G(y)&=c  \delta^{-2+\frac43 it}  \int_{\R \times \mathbb{Z}} \int_{\R} e^{2\pi i \delta^{-2} s (y-w)\cdot y } \widehat{\phi}(\frac{s}{T\delta^2}) \widehat{a}_{-\frac{1}{2}+it}( s) G(w) \widetilde{\chi}(y,w) {\rm d} s {\rm d} w.
\end{align*}
 So we are done if we could prove the following $L^2$ estimate.
\begin{proposition}\label{prop:L2 estimate}
Suppose $1 \ll \delta \ll T^{-1}, \theta \in \mathcal{C}_\delta$ with $dist((0,0),\theta)\gg \delta$. If a smooth function $b$ has order $\frac12$ and satisfies $\supp(b)\subset [-\frac{T\delta^2}{4}, \frac{T\delta^2}{4}]$. Then the operator
\begin{align*}
V(G)(y)&=  \delta^{-2}  \int_{\R \times \mathbb{Z}} \int_{\R} e^{2\pi i \delta^{-2} s (y-w)\cdot y } b(s) G(w)  \widetilde{\chi}(y,w) {\rm d} s {\rm d} w
\end{align*}
maps $L^2(\R\times \mathbb{Z})$ to $ L^2(\R\times \mathbb{Z})$, where 
$$\widetilde{\chi}(y,w)=  \chi_{\theta}(y)\chi_{\theta}(w)   $$
and $\chi_{\theta}$ is a smooth cut-off function on $\theta$.
\end{proposition}

\section{Proof of $L^{2}$ estimate}
In this section, we prove Proposition \ref{prop:L2 estimate} to complete the proof of our main theorem. To prove Proposition \ref{prop:L2 estimate} in Euclidean case, refer to \cite{tao1998bilinear}, the authors use the classical results about Fourier integral operators (see \cite{1971Fourier}), maybe this method is not applicable to the waveguide case. We use the almost orthogonality principle and the Schur's lemma to replace this part as in \cite{tao1998bilinear}. We refer to \cite{1990Averages} for the similar discussion.

In the proof of Proposition \ref{prop:L2 estimate}, we will need a lemma to achieve the same effect as one-dimensional oscillatory integral.

\begin{lemma} \label{lem-1}
Let $\kappa \in C_c^\infty(\R) $ and $P\in C^\infty(\R)$.

(1) If $|P^\prime(\delta x)|  \sim \delta^{-1} \lambda \ll 1$ for any $x \in \supp(\kappa)$, where $\lambda \ge 1$ is a constant. Then
\begin{equation}\label{lem-1-eq-1}
|\sum_n e^{2\pi i P(n)} \kappa(\frac{n}{\delta})|\less_{N,\kappa} \lambda^{-N} \delta
\end{equation}
holds for any $N \in \mathbb{N}^*$.

(2) If $|P^\prime(\delta x)|\ll 1$ for any $x \in \supp(\kappa)$. Then
\begin{equation}\label{lem-1-eq-2}
\lf |\sum_n (e^{2\pi i P(\cdot)})^{(N)}(n) \kappa (\frac{n}{\delta})\ri| \less_{N,\kappa} \delta^{-N+1}
\end{equation}
holds for any $N \in \mathbb{N}^*$.
\end{lemma}

\begin{proof}
Let $F(\xi)=e^{2\pi i P(\xi)} \kappa(\frac{\xi}{\delta})$.

We first prove (\ref{lem-1-eq-1}). Using Poisson's summation formula,  (\ref{lem-1-eq-1}) reduces to
$$| \sum_m \widehat{F}(m)|\less_{N,\kappa} \lambda^{-N} \delta.$$
Using one-dimensional oscillatory integral, there holds
$$ |\widehat{F}(0)|=|\int e^{2\pi i P(\xi)} \kappa(\frac{\xi}{\delta}) {\rm d} \xi| \less_{N,\kappa} \lambda^{-N} \delta,$$
and
\begin{equation}\label{lem-1-eq-4}
 |\widehat{F}(m)|=|\int e^{2\pi i (P(\xi)-m\xi)} \kappa(\frac{\xi}{\delta}) {\rm d} \xi| \less_{N,\kappa} |m\delta|^{-N }\delta ,
\end{equation}
where $|m|\ge 1$.

Observe that we assume that $\delta^{-1} \lambda \ll 1$, so (\ref{lem-1-eq-1}) holds.

We now turn to prove (\ref{lem-1-eq-2}). By induction, we only need to prove that
\begin{equation}\label{lem-1-eq-3}
|\sum_n F^{(N)}(n)| \less_{N,\kappa} \delta^{-N+1}.
\end{equation}
Similar as (\ref{lem-1-eq-1}), we use Poisson's summation formula, then  (\ref{lem-1-eq-3}) reduces to
\begin{equation*}
|\sum_m m^N \widehat{F}(m) |\less_{N,\kappa} \delta^{-N+1}.
\end{equation*}
By (\ref{lem-1-eq-4}) and we replace $N$ by $N+2$, then there holds
$$ |m^N \widehat{F}(m)| \less_{N,\kappa} m^{-2} \delta^{-(N+1) } ,$$
when $|m|\ge 1$, so
$$\sum_m |m^N \widehat{F}(m) |\less_{N,\kappa} \delta^{-N+1}$$
holds, we complete the proof.
\end{proof}

\begin{remark}\label{rmk-1}
By using oscillation integral estimation and $N$ times integral by parts respectively, we know that if we replace the summation of the integer variable $n\in \mathbb{Z}$ by the integral of real variable $x\in \R$, the estimates in Lemma \ref{lem-1} also hold, and in this time, we don't need the assumption $|P^\prime(\delta x)|\ll 1$.
\end{remark}

Now, we turn to complete the proof of Proposition \ref{prop:L2 estimate}.

\noindent{\em Proof of Proposition \ref{prop:L2 estimate}.} We will assume that $dist(0, P_{\mathbb{Z}}\theta)\gg\delta$, and in the case $dist(0, P_{\mathbb{R}}\theta)\gg \delta$, we only need to use Remark \ref{rmk-1} instead of Lemma \ref{lem-1}.
 Define
\begin{align*}
V_k(G)(y)&=  \delta^{-2}  \int_{\R \times \mathbb{Z}} \int_{\R} e^{2\pi i \delta^{-2} s (y-w)\cdot y } b(s) \phi_k(s) G(w) \widetilde{\chi}(y,w) {\rm d} s {\rm d} w,
\end{align*}
where $0 \le k \less \log_2 (T\delta^2), \phi_0(s)=\phi(s)$ and $\phi_j(s)=\phi(\frac{s}{2^j})-\phi(\frac{s}{2^{j-1}})$ when $j\ge 1$. Denote $y=(y_1,y_2)\in \R \times \mathbb{Z}, w=(w_1,w_2)\in \R \times \mathbb{Z}$.

The kernel of $V_j^* V_k$ is
\begin{align*}
   & K_{k,j}(w,w^\prime) \\
   =& \delta^{-4} \int_{\R \times \mathbb{Z}} \int_{\R} \int_{\R} e^{2\pi i \delta^{-2} (s (y-w)\cdot y -s^\prime (y-w^\prime) \cdot y) }b(s) \phi_k(s) \overline{b(s^\prime) \phi_j(s^\prime)}\widetilde{\chi}(y,w) \widetilde{\chi}(y,w^\prime) {\rm d} s {\rm d} s^\prime {\rm d} y.
\end{align*}
When $|k-j|\gg 1$. Note that
$$    |\partial_{y_2}(\delta^{-2} (s (y-w)\cdot y -s^\prime (y-w^\prime) \cdot y))|\sim 2^k \delta^{-1} \ll 1, y_2 \in P_{\mathbb{Z}}\theta,$$ 
we could control $|K_{k,j}(w,w^\prime)|$ by C$\delta^{-2} 2^{-\frac32 |k-j|}$, by considering the summation of $y_2$ and  using Lemma \ref{lem-1}(1). So by Schur's lemma, we conclude that $\|V_j^* V_k\|_{L^2 \to L^2}\less 2^{-\frac32 |k-j|}$.

When $|k-j|\less 1$, we assume $k=j$ without loss of generality.  Changes the variable $s^\prime=s+h$, then we write $K_{k,k}(w,w^\prime)$ as
\begin{align*}
& \int_{\R}\Big(   \delta^{-4} \int_{\R \times \mathbb{Z}} \int_{\R} e^{2\pi i \delta^{-2} (s (w^\prime-w)\cdot y -h (y-w^\prime) \cdot y) }b(s) \phi_k(s) \overline{b(s+h) \phi_k(s+h)} \\
& \cdot \widetilde{\chi}(y,w) \widetilde{\chi}(y,w^\prime) {\rm d} s {\rm d} y \Big){\rm d} h \\
:=& \int_{\R} \widetilde{K}_{k,k}(w,w^\prime,h) {\rm d} h,
\end{align*}
and define
$$S_{w,w^\prime,h}(y,s)= s (w^\prime-w)\cdot y -h (y-w^\prime) \cdot y.$$
Now we treat $S$ as the function of $y$ and $s$, and we calculate the first partial derivative of $e^{2\pi i \delta^{-2} S }$, one has that
$$ \partial_{s}(e^{2\pi i \delta^{-2} S })=2\pi i \delta^{-2} e^{2\pi i \delta^{-2} S }\lf( (w-w^\prime)\cdot y  \ri), $$
$$ \partial_{y_1}(e^{2\pi i \delta^{-2} S })=2\pi i \delta^{-2} e^{2\pi i \delta^{-2} S}\lf( s(w_1^\prime-w_1)-h(2y_1-w_1^\prime)   \ri) , $$
$$ \partial_{y_2}(e^{2\pi i \delta^{-2} S })=2\pi i \delta^{-2} e^{2\pi i \delta^{-2} S }\lf( s(w_2^\prime-w_2)-h(2y_2-w_2^\prime)   \ri),$$
thus
\begin{align*}
&
\begin{pmatrix}
  h  \\
  w_1^\prime-w_1\\
  w_2^\prime-w_2
\end{pmatrix} e^{2\pi i \delta^{-2} S} \\
=&  (2\pi i \delta^{-2} s (2y-w^\prime)\cdot y)^{-1} \cdot \\
&\begin{pmatrix}
  s^2\partial_s(e^{2\pi i \delta^{-2} S})-y_1 s \partial_{y_1}(e^{2\pi i \delta^{-2} S}) -y_2 s \partial_{y_2}(e^{2\pi i \delta^{-2} S}) \\
  (2y_1-w_1^\prime)s\partial_s(e^{2\pi i \delta^{-2} S})+(2y_2-w_2^\prime)y_2 \partial_{y_1}(e^{2\pi i \delta^{-2} S}) -(2y_1-w_1^\prime)y_2 \partial_{y_2}(e^{2\pi i \delta^{-2} S})   \\
  (2y_2-w_2^\prime)s\partial_s(e^{2\pi i \delta^{-2} S})-(2y_2-w_2^\prime)y_1 \partial_{y_1}(e^{2\pi i \delta^{-2}S}) +(2y_1-w_1^\prime)y_1 \partial_{y_2}(e^{2\pi i\delta^{-2} S})
\end{pmatrix}.
\end{align*}
Then, integral by parts and Lemma \ref{lem-1}(2) conclude that
\begin{equation}\label{bilinear estimate-thm-eq-4}
|K(w,w^\prime)|\le \int_{\R} |\widetilde{K}_{k,k}(w,w^\prime,h)| {\rm d} h\less \delta^{-2} (2^k)^{2}\int_\R (1+|h|+2^k\delta^{-1}|w-w^\prime|)^{-N} {\rm d} h ,
\end{equation}
and then by Schur's lemma, we conclude that $\|V_k^* V_k\|_{L^2 \to L^2}\less 1$. Thus we complete the proof of Proposition \ref{prop:L2 estimate} because of the Cotlar-Knapp-Stein almost orthogonality lemma.

In the end of proof, we explain that how to get (\ref{bilinear estimate-thm-eq-4}). First, we have the trivial bound without oscillatory integral, that is 
\begin{align*}
  &|\widetilde{K}_{k,k}(w,w^\prime,h)|\\
  \less & \delta^{-4} \int_{\R \times \mathbb{Z}} \int_{\R} |b(s) \phi_k(s) \overline{b(s+h) \phi_k(s+h)} \widetilde{\chi}(y,w) \widetilde{\chi}(y,w^\prime)| {\rm d} s {\rm d} y \less \delta^{-2} (2^k)^{2}.
\end{align*}
Next we should consider the oscillatory term. We use the identity
\begin{align*}
   &  e^{2\pi i \delta^{-2} S}\\
=&h^{-1} (2\pi i \delta^{-2}  (2y-w^\prime)\cdot y)^{-1} \lf(   s^2\partial_s(e^{2\pi i \delta^{-2} S})-y_1 s \partial_{y_1}(e^{2\pi i \delta^{-2} S}) -y_2 s \partial_{y_2}(e^{2\pi i \delta^{-2} S})    \ri)
\end{align*}
$N$ times, then we have
$$ e^{2\pi i \delta^{-2} S}=h^{-N} \sum_{m,n,l\in \mathbb{N}:m+n+l=N} F_{m,n,l,w,w^\prime}(y,s) \partial_s^m \partial_{y_1}^n \partial_{y_2}^l  (e^{2\pi i \delta^{-2} S}).$$
We substitute the above equation into $\widetilde{K}_{k,k}(w,w^\prime,h)$. For fixed $m,n,l$, we use integral by parts of $s,y_1$ respectively $m,n$ times, then the remaining part is about the summation of $y_2 \in \mathbb{Z}$. Thus we could use Lemma \ref{lem-1}(2) to control the remaining part by $\delta^{-2} (2^k)^{2} |h|^{-N}$, because $|\partial_{y_2}(\delta^{-2} S)|\less |s| \delta^{-1}\less T \delta \ll 1$.

Similarly, we can control $|\widetilde{K}_{k,k}(w,w^\prime,h)|$ by $\delta^{-2} (2^k)^{2} (2^k\delta^{-1}|w-w^\prime|)^{-N}$.

\hfill$\Box$\vspace{2ex}

\appendix
\section{Some further discussions}
For $0< T \less \delta^{-2}$ and $f,g$ which are "$\delta-$transverse", we have the stronger estimate
$$ \| e^{it\Delta}f \cdot  e^{it\Delta} g    \|_{L^2([0,T] \times \mathbb{R} \times \mathbb{T})}\lesssim (T^{\frac12}\delta) \delta^{2-\frac{4}{p}}   \|\widehat{f} \|_{L^p(\mathbb{R} \times \mathbb{Z})} \|\widehat{g} \|_{L^p(\mathbb{R} \times \mathbb{Z})},\quad p=\frac{12}{7}.$$
The proof relies on the calculation in Section 2, and we give a brief proof here. As in Section 2, it suffices to show that
$$ \|U(G)\|_{L^3} \less \|G\|_{L^{\frac32}},$$
where 
$$ U(G)(y)=\delta^{-\frac43} \int_{\R \times \mathbb{Z}}  \phi\lf( T y\cdot (y-w) \ri)\widetilde{\chi}(y,w) G(w){\rm d} w.$$
Then by the H\"older inequality,
\begin{align*}
  &\|U(G)\|_{L^3} \less \delta^{\frac{2}{3}} \|U(G)\|_{L^\infty}  
 \less  \delta^{-\frac23} \|G\|_{L^{\frac32}} \| \phi\lf( T y\cdot (y-w) \ri)\chi_\theta(w)\|_{L^3_w L^\infty_y} \\
 \less & \delta^{-\frac23} \|G\|_{L^{\frac32}} \|\chi_\theta(w)\|_{L^3_w L^\infty_y} \less \|G\|_{L^\frac23}.
\end{align*}

In the fact, as our expected goal, we wish to prove the estimate
\begin{equation}\label{eq-wish}
  \| e^{it\Delta}f \cdot  e^{it\Delta} g    \|_{L^2([0,\delta^{c(p)}] \times \mathbb{R} \times \mathbb{T})}\lesssim \delta^{2-\frac{4}{p}}   \|\widehat{f} \|_{L^p(\mathbb{R} \times \mathbb{Z})} \|\widehat{g} \|_{L^p(\mathbb{R} \times \mathbb{Z})}
\end{equation}
for all $p\in [\frac{12}{7},2]$, where $c(p)$ is a function of $p$ with $c(p)=-1$ when $p=\frac{12}{7}$ and $c(p)=1$ when $p=2$. See Appendix 2,  we guess that 
$$c(p)=8\cdot \frac{p-2}{4-p}.$$
Interestingly, by the Strichartz estimate on $\R \times \mathbb{Z}$ which appears in \cite{takaoka20012d} and Theorem \ref{thm: main},  (\ref{eq-wish}) holds for $p=\frac{12}{7}$ and $p=2$, but we failed to find an interpolation method to get (\ref{eq-wish}) for $\frac{12}{7}<p<2$.

\section{A (counter) example }
In this appendix, we explain that, the requirement $T\delta \less 1$ in Theorem \ref{thm: main} is necessary. Furthermore, we can get a necessary condition for the inequality (\ref{eq: main}) to hold. Because of the stronger estimate in Appendix A, we assume that $T \gg \delta^{-2}$.

Let $\theta_1=[0,\delta] \times \{0,1,\cdots,\delta\}, \theta_2=[0,\delta] \times \{100\delta,100\delta+1,\cdots,101\delta\}$. We choose a nonnegative function $\varphi\in C_c^\infty([-1,1])$ with $\widehat{\varphi}$ is nonnegative and $\widehat{\varphi}\ge 1_{[-1,1]}$.

As the calculation in Section 2, one has
\begin{align*}
&\| e^{i \Delta} f \cdot  e^{i \Delta} g    \|_{L^2([-T,T] \times \R \times \mathbb{T})}^2 \\
\gea &\|\sqrt{\varphi(t/T)} e^{i \Delta} f \cdot e^{i \Delta} g    \|^2_{L^2(\R \times \R \times \mathbb{T})} \\
=&\int_\R \int_{\R \times \mathbb{T}} \varphi(\frac{t}{T}) e^{i \Delta} f \cdot e^{i \Delta} g \cdot \overline{e^{i \Delta} f} \cdot \overline{e^{i \Delta} g} \\
=& T \int_{(\R\times \mathbb{Z})^3} \widehat{\varphi}(2T(y-x)\cdot (y-w) )\widehat{f}(x) \overline{\widehat{f}(x+y-w)}  \widehat{g}(y) \overline{\widehat{g}(w)} {\rm d}x {\rm d}y {\rm d}w.
\end{align*}
Now we choose a small constant $c>0$, and let $\widehat{f}=1_{[0,cT^{-\frac12}]\times \{0\}}, \widehat{g}=1_{[0,cT^{-\frac12}]\times \{100\delta\}}$. Thus, for $p \in [\frac{12}{7},2]$, (\ref{eq: ls}) holds unless
$$ T \cdot (T^{-\frac12})^3 \less (\delta^{2-\frac4p} T^{-\frac1p})^2,$$
that is
$$ T \less \delta^{8\cdot \frac{p-2}{4-p}},$$
especially, when $p=\frac{12}{7}$, that is
$$ T \less \delta^{-1}.$$

 \bibliographystyle{amsplain}
\bibliography{BG.bib}
\bibliographystyle{plain}
\end{document}